\documentclass{conm-p-l}
\usepackage{amssymb, euscript}
\newtheorem{theorem}{Theorem}[section]
\newtheorem{lemma}[theorem]{Lemma}
\newtheorem{corollary}[theorem]{Corollary}
\theoremstyle{definition}
\newtheorem{definition}[theorem]{Definition}

\theoremstyle{remark}
\newtheorem{remark}[theorem]{Remark}

\numberwithin{equation}{section}

\begin{document}

\title{Approximations in $L_{p}$-norms and Besov spaces on compact manifolds}

\author{Isaac Z. Pesenson }\footnote{ Department of Mathematics, Temple University,
 Philadelphia,
PA 19122; pesenson@temple.edu }


\begin{abstract}
The  objective of the paper is to describe Besov spaces  on general compact Riemannian manifolds in terms of the best approximation by eigenfunctions of elliptic differential operators. 
\end{abstract}

\maketitle

\section{Introduction}

Approximation theory and its relations to function spaces on Riemannian manifolds is an old subject which still attracts  attention of mathematicians \cite{brdai}, \cite{FM1}, \cite{gpa}, \cite{NRT}, \cite{Pes1}.

The goal  of the paper is to give an alternative proof of a theorem in \cite{gp} (see Theorem \ref{ApproxTh}  below ) which  characterizes  functions in  Besov spaces on compact Riemannian manifolds by best approximations by eigenfunctions of elliptic differential operators. This theorem was an important ingredient of a construction which led to a descriptions of Besov spaces in terms of bandlimited localized frames (see \cite{gp} for details). 

Let $\mathbf{M}$, dim$\>\mathbf{M}=n$, be a connected compact Riemannian
manifold without boundary and the corresponding space $L_{p}({\bf M}), 1\leq p\leq\infty,$ is constructed by using Riemannian measure.  
To define Sobolev spaces,  we fix a  covering  $\{B(y_{\nu}, r_{0})\}$ of $\bold{M}$ of finite
multiplicity by balls $B(y_{\nu}, r_{0})$ centered at $y_{\nu}\in \bold{M}$ of radius
$r_{0}< \rho_{\bold{M}}$, where  $\rho_{\bold{M}}$ is the injectivity radius of the manifold.
For a  fixed partition of unity
$\Psi=\{\psi_{\nu}\}$ subordinate to this covering  the Sobolev
spaces $W^{k}_{p}(\bold{M}), k\in \mathbb{N}, 1\leq p<\infty,$ are
introduced as the completion of $C^{\infty}(\bold{M})$ with respect
to the norm
\begin{equation}
\|f\|_{W^{k}_{p}(\bold{M})}=\left(\sum_{\nu}\|\psi_{\nu}f\|^{p}
_{W^{k}_{p}(\mathbf{R}^{n})}\right) ^{1/p}.\label{Sobnorm}
\end{equation}

Similarly,

\begin{equation}
\|f\|_{B^{\alpha}_{p,q}(\bold{M})}=\left(\sum_{\nu}\|\psi_{\nu}f\|^{p}
_{B^{\alpha}_{p,q}(\mathbf{R}^{n})}\right) ^{1/p},\label{Besnorm}
\end{equation}
where $B^{\alpha}_{p,q}(\mathbf{R}^{n})$ is the Besov space $B^{\alpha}_{p,q}(\bold{M}), \alpha>0, 1\leq p<\infty, 0\leq q\leq \infty.$
It is known \cite{Tr} that such defined spaces are independent on the choice of a partition of unity.

Let $L$ be a  second order elliptic differential operator on ${\bf M}$ with smooth coefficients. By duality such an operator can be extended  to the space of distributions on ${\bf M}$. This extension, with domain consisting of all $f\in L_{2}({\bf M})$ for which $Lf\in L_{2}({\bf M})$ defines an operator in the space $L_{2}({\bf M})$. 
We will assume that this operator is self-adjoint and non-negative.

One can consider the positive square root $L^{1/2}$ and by duality extend it to the space of distributions on ${\bf M}$. The corresponding operator $L^{k/2},\>k\in \mathbb{N},$ in a space  $L_{p}({\bf M}), 1\leq p\leq\infty, $ is defined on the set of all distributions $f\in L_{p}({\bf M})$ for which $L^{k/2}f\in L_{p}({\bf M})$. It is known that the domain of  the operator $L^{k/2}$  is exactly the Sobolev space $W^{k}_{p}(\bold{M})$. 

In  every space $L_{p}({\bf M}), \>1\leq p\leq \infty, $ such defined operators (which will  all be denoted as $L$) have the same spectrum $0=\lambda_{0}<\lambda_{1}\leq \lambda_{2}\leq ...$ and the same set of eigenfunctions.
 Let
$u_{0}, u_{1}, u_{2}, ...$ be a corresponding complete set of eigenfunctions which are orthonormal in the space $L_{p}({\bf M})$.   

The notation 
$\mathbf{E}_{\omega}(L),\ \omega>0,$ will be used for the span of all
eigenfunctions of $L$, whose corresponding eigenvalues
are not greater than $\omega$.

For $1 \leq p \leq \infty$, if $f \in L_p({\bf M})$, we let
\begin{equation}
\mathcal{E}(f,\omega,p)=\inf_{g\in
\textbf{E}_{\omega}(L)}\|f-g\|_{L_{p}({\bf M})}.
\end{equation}

The following theorem was proved in \cite{gp}.
\begin{theorem}
\label{ApproxTh}
If  $\alpha  > 0$, $1 \leq p \leq \infty$, and $0 < q \leq  \infty$ then
$f \in B^{\alpha }_{p, q}({\bf M})$ if and only if $f \in L_p(\bold{M})$ and
\begin{equation}
\label{errgdpq}
\|f\|_{\mathcal{A}^{\alpha}_{ p, q}({\bf M})} := 
\|f\|_{L_p(\bold{M})} + \left(\sum_{j=0}^{\infty} (2^{\alpha j}{\mathcal E}(f, 2^{2j},p))^q \right)^{1/q} < \infty.
\end{equation}
Moreover,
\begin{equation}
\label{errnrmeqway}
\|f\|_{\mathcal{A}^{\alpha }_{p, q}({\bf M})} \sim \|f\|_{B^{\alpha }_{p, q}({\bf M})}.
\end{equation}
\end{theorem}

Our objective is to give a proof of this Theorem which is different from the proof presented in \cite{gp}. The new proof relies  on powerful tools of the theory of interpolation of linear operators.

\section{Kernels and Littlewood-Paley decomposition on compact Riemannian manifolds}

Assuming  $F \in C_{0}^{\infty}(\bf{R})$ and using the spectral theorem,
one can define the bounded operator $F(t^{2}L)$ on $L_2({\bf M})$.  In fact, for $f \in L_2({\bf M})$,
\begin{equation}
\label{ft2F}
[F(t^{2}L)f](x) = \int K^{F}_t(x,y) f(y) dy,
\end{equation}
where 
\begin{equation}
\label{expout}
K^{F}_t(x,y) = \sum_l F(t^2\lambda_l)u_l(x)u_l(y).
\end{equation}
 We call $K^{F}_t$ the kernel of $F(t^2L)$.  $F(t^2L)$ maps $C^{\infty}({\bf M})$
to itself continuously, and may thus be extended to be a map on distributions.  In particular
we may apply $F(t^2L)$ to any $f \in L_p({\bf M}) \subseteq L_1({\bf M})$ (where $1 \leq p \leq \infty$), and by Fubini's theorem
$F(t^2L)f$ is still given by (\ref{ft2F}).  

The following Theorem  about $K^{F}_t$ was proved in \cite{gp}  for general elliptic second order differential self-adjoint positive operators.
\begin{theorem}
\label{nrdglc}
 Assume $F\in C_{0}^{\infty}(\bf{R})$  and  let $K^{F}_t(x,y)$ be the kernel of $F(t^{2}L)$.  
Then for any $N>n$ there exists  $C(F,N) > 0$ such that
\begin{equation}
\label{kersize}
|K^{F}_t(x,y)| \leq \frac{C(F, N)}{t^{n}\left[ 1+\frac{d(x,y)}{t} \right]^{N}},\>\>\>n=dim\>{\bf M},
\end{equation}
for $0<t\leq 1$ and all $x,y \in {\bf M}$. 
The constant $C(F, N)$ depends on the norm of $F$  in the space $C^{N}(\bf{R})$. 

\end{theorem}

This estimate has the following important implication.

\begin{corollary}
\label{Kerbound}
Consider  $1 \leq \alpha \leq \infty$, with conjugate index $\alpha'$.
In the situation of Theorem \ref{kersize}, there is a constant $C > 0$ such that 
\begin{equation}
\label{kint3a}
\left(\int |K^{F}_t(x,y)|^{\alpha}dy\right)^{1/\alpha} \leq Ct^{-n/\alpha'} \ \ \ \ \ \ \ \ \ \ \ \ \ \ \ \ \ \ \ \ \mbox{for all } x.
\end{equation}
\end{corollary}
\begin{proof}  First, we note 
that if $N > n$, $x \in {\bf M}$ and $t > 0$, then
\begin{equation}
\label{intest}
\int_{\bf M} \frac{1}{\left[1 + (d(x,y)/t)\right]^N} dy \leq Ct^{n},\>\>\>n= \dim {\bf M},
\end{equation}
with $C$ independent of $x$ or $t$. 
Indeed, there exist $c_{1}, c_{2}>0$ such that for all $x\in M$ and all sufficiently small $r\leq \delta$ one has
 $$
 c_{1}r^{n}\leq |B(x,r)|\leq c_{2}r^{n},
 $$
 and if $r>\delta$
 $$
 c_{3}\delta^{n}\leq |B(x,r)|\leq |\mathbf{M}|\leq c_{4}r^{n}.
 $$

 For fixed $x,t$  let $A_{j}=B(x, 2^{j}t)\setminus B(x, 2^{j-1}t)$. Then $|A_{j}|\leq c_{4}2^{nj}t^{n}$ and for every $A_{j}$ one has
 $$
\int_{A_{j}} \frac{1}{\left[1 + (d(x,y)/t)\right]^N} dy\leq c_{4}2^{(n-N)j}t^{n}.
 $$
 In other words,
 $$
 \int_{\bf M} \frac{1}{\left[1 + (d(x,y)/t)\right]^N} dy=\sum_{j}\int_{A_{j}} \frac{1}{\left[1 + (d(x,y)/t)\right]^N} dy\leq c_{4}\sum_{j}2^{(n-N)j}t^{n}.
 $$
 Using this estimate and (\ref{kersize}) one obtains (\ref{intest}).

This completes the proof.

\end{proof}

  \begin{theorem} \label{boundness}
 If $F\in C_{0}^{\infty}(\mathbf{R})$ and  $(1/q)+1 = (1/p)+(1/\alpha)$ then for the same constant $C$ as in (\ref{kint3a}) one has 
 $$
 \|F(t^{2}L)\|_{L_{p}(\mathbf{M})\rightarrow L_{q}(\mathbf{M})}\leq Ct^{-n/\alpha'},\>\>\>n=dim\>\mathbf{M},
 $$
  for all $0<t\leq 1$. In particular, 
 $$
 \|F(t^{2}L)\|_{L_{p}(\mathbf{M})\rightarrow L_{p}(\mathbf{M})}\leq C,\>\>\>n=dim\>\mathbf{M},
 $$
  for all $0<t\leq 1$.

 \end{theorem}
 
\begin{proof} The proof follows from Corollary \ref{Kerbound} and the following Young inequalities.
\end{proof}
 \begin{lemma}
\label{younggen}
Let $\mathcal{K}(x,y)$ be a measurable function on $\mathbf{M}\times\mathbf{M}$. Suppose that $1 \leq p, \alpha \leq \infty$, and that $(1/q)+1 = (1/p)+(1/\alpha)$. 
If there exists a  $C > 0$ such  that
\begin{equation}
\label{kint1}
\left(\int_{\mathbf{M}} |\mathcal{K}(x,y)|^{\alpha}dy\right)^{1/\alpha} \leq C \ \ \ \ \ \ \ \ \ \mbox{for all } x\in \mathbf{M}, 
\end{equation}
and
\begin{equation}
\label{kint2}
\left(\int_{\mathbf{M}} |\mathcal{K}( x,y)|^{\alpha}dx\right)^{1/\alpha} \leq C \ \ \ \ \ \ \ \ \ \ \mbox{for all } y\in \mathbf{M},
\end{equation}
then for the same constant  $C$ for all $f \in L_p(\mathbf{M})$  one has the inequality
$$
\left\|\int_{\mathbf{M}}\mathcal{K}( x,y)f(y)dy\right\|_{L_{q}({\bf M})} \leq C\|f\|_{L_{p}({\bf M})}.
$$
\end{lemma}

\begin{proof} Let $\beta = q/\alpha \geq 1$,
so that $\beta' = p'/\alpha$.  For any $x$, we have
\begin{eqnarray*}
|({\mathcal K}f)(x)| & \leq & \int |{\mathcal K}(x,y)|^{1/\beta'}|{\mathcal K}(x,y)|^{1/\beta}f(y)|dy  \\
& \leq & \left(\int |{\mathcal K}(x,y)|^{p'/\beta'}dy\right)^{1/p'} \left(\int |{\mathcal K}(x,y)|^{p/\beta}|f(y)|^p dy\right)^{1/p}\\
& \leq & c^{1/\beta'}\left(\int |{\mathcal K}(x,y)|^{p/\beta}|f(y)|^p dy\right)^{1/p}
\end{eqnarray*}
since $p'/\beta' = \alpha$, $\alpha/p' = 1/\beta'$.  Thus
\begin{eqnarray*}
\|{\mathcal K}f\|^p_q & \leq & c^{p/\beta'}\left(\int \left(\int |{\mathcal K}(x,y)|^{p/\beta}|f(y)|^p dy \right)^{q/p}dx \right)^{p/q} \\
& \leq  & c^{p/\beta'}\int \left(\int |{\mathcal K}(x,y)|^{pq/\beta p}|f(y)|^{pq/p} dx \right)^{p/q}dy \\
& = & c^{p/\beta'}\int \left(\int |{\mathcal K}(x,y)|^{\alpha}dx \right)^{p/q}|f(y)|^p dy \\
& \leq & c^{p/\beta'}c^{p/\beta}\|f\|^p_p
\end{eqnarray*}
as desired.  (In the second line, we have used Minkowski's inequality for integrals.)\\

\end{proof}

\section{Interpolation and  Approximation spaces. }\label{interp}

The goal of this section is to remind  the reader of certain  connections
between interpolation spaces and approximation spaces which will
be used later. The general theory of interpolation spaces can be found
in \cite{BL}, \cite{BB},  \cite{KPS}. The notion of
Approximation spaces and their relations to Interpolations spaces
can be found in \cite{BS}, \cite{PS}, and  in \cite{BL}, Ch. 3 and 7. 

It is important to realize that relations between Interpolation and Approximation spaces cannot  be described in  the language of normed spaces. One has to use the language of
quasi-normed linear spaces   to treat
 interpolation and approximation spaces simultaneously.

Let $E$ be a linear space. A quasi-norm $\|\cdot\|_{E}$ on $E$ is
a real-valued function on $E$ such that for any $f,f_{1}, f_{2}\in
E$ the following holds true

\begin{enumerate}
\item $\|f\|_{E}\geq 0;$

\item $\|f\|_{E}=0  \Longleftrightarrow   f=0;$

\item $\|-f\|_{E}=\|f\|_{E};$

\item $\|f_{1}+f_{2}\|_{E}\leq C_{E}(\|f_{1}\|_{E}+\|f_{2}\|_{E}),
C_{E}>1.$

\end{enumerate}

We say that two quasi-normed linear spaces $E$ and $F$ form a
pair, if they are linear subspaces of a linear space $\mathcal{A}$
and the conditions 
$$
\|f_{k}-g\|_{E}\rightarrow 0,\>\>\>\>\|f_{k}-h\|_{F}\rightarrow 0, \>\>k\rightarrow \infty,\>\>f_{k},\> g, \>h \in \mathcal{A}, 
$$
imply equality $g=h$. For a such pair $E,F$ one can construct a
new quasi-normed linear space $E\bigcap F$ with quasi-norm
$$
\|f\|_{E\bigcap F}=\max\left(\|f\|_{E},\|f\|_{F}\right)
$$
and another one $E+F$ with the quasi-norm
$$
\|f\|_{E+ F}=\inf_{f=f_{0}+f_{1},f_{0}\in E, f_{1}\in
F}\left(\|f_{0}\|_{E}+\|f_{1}\|_{F}\right).
$$

All quasi-normed spaces $H$ for which $E\bigcap F\subset H \subset
E+F$ are called intermediate between $E$ and $F$. A vector space 
homomorphism $T: E\rightarrow F$ is called bounded if
$$
\|T\|=\sup_{f\in E,f\neq 0}\|Tf\|_{F}/\|f\|_{E}<\infty.
$$
One says that an intermediate quasi-normed linear space $H$
interpolates between $E$ and $F$ if every bounded homomorphism $T:
E+F\rightarrow E+F$ which is also bounded when restricted to $E$ and $F$ is  bounded
homomorphism of $H$ into $H$.

On $E+F$ one considers the so-called Peetre's $K$-functional
\begin{equation}
K(f, t)=K(f, t,E, F)=\inf_{f=f_{0}+f_{1},f_{0}\in E, f_{1}\in
F}\left(\|f_{0}\|_{E}+t\|f_{1}\|_{F}\right).\label{K}
\end{equation}
The quasi-normed linear space $(E,F)^{K}_{\theta,q}, 0<\theta<1,
0<q\leq \infty,$ or $0\leq\theta\leq 1,  q= \infty,$ is introduced
as a set of elements $f$ in $E+F$ for which
\begin{equation}
\|f\|_{\theta,q}=\left(\int_{0}^{\infty}
\left(t^{-\theta}K(f,t)\right)^{q}\frac{dt}{t}\right)^{1/q}.\label{Knorm}
\end{equation}

It turns out that $(E,F)^{K}_{\theta,q}, 0<\theta<1, 0\leq q\leq
\infty,$ or $0\leq\theta\leq 1,  q= \infty,$ with the quasi-norm
(\ref{Knorm})  interpolates between $E$ and $F$. The following
Reiteration Theorem is one of the main results of the theory (see
\cite{BL}, \cite{BB}, \cite{KPS}, \cite{PS}).
\begin{theorem}\label{Int}
Suppose that $E_{0}, E_{1}$ are complete intermediate quasi-normed
linear spaces for the pair $E,F$.  If $E_{i}\in
\mathcal{K}(\theta_{i}, E, F)$, which means
$$
K(f,t,E,F)\leq Ct^{\theta_{i}}\|f\|_{E_{i}}, i=0,1,
$$
where $ 0\leq \theta_{i}\leq 1, \theta_{0}\neq\theta_{1},$ then
$$
(E_{0},E_{1})^{K}_{\eta,q}\subset (E,F)^{K}_{\theta,q},
$$
where $0<q<\infty, 0<\eta<1,
\theta=(1-\eta)\theta_{0}+\eta\theta_{1}$.

 If for the same pair
$E,F$ and the same $E_{0}, E_{1}$  one has $E_{i}\in
\mathcal{J}(\theta_{i}, E, F)$, i. e.
$$
\|f\|_{E_{i}}\leq C\|f\|_{E}^{1-\theta_{i}}\|f\|_{F}^{\theta_{i}},
i=0,1,
$$
where $ 0\leq \theta_{i}\leq 1, \theta_{0}\neq\theta_{1},$ then
$$
(E,F)^{K}_{\theta,q},\subset (E_{0},E_{1})^{K}_{\eta,q},
$$
where $0<q<\infty, 0<\eta<1,
\theta=(1-\eta)\theta_{0}+\eta\theta_{1}$.

\end{theorem}

It is important to note that in all cases  considered
in the present article the space $F$ will be continuously embedded
as a subspace into $E$. In this case (\ref{K}) can be introduced
by the formula
$$
K(f, t)=\inf_{f_{1}\in
F}\left(\|f-f_{1}\|_{E}+t\|f_{1}\|_{F}\right),
$$
which implies the inequality
\begin{equation}
K(f, t)\leq \|f\|_{E}.
\end{equation}
This inequality can be used to show  that the norm (\ref{Knorm})
is equivalent to the norm
\begin{equation}
\|f\|_{\theta,q}=\|f\|_{E}+\left(\int_{0}^{\varepsilon}
\left(t^{-\theta}K(f, t)\right)^{q}\frac{dt}{t}\right)^{1/q},
\varepsilon>0,
\end{equation}
for any positive $\varepsilon$.

Let us introduce another functional on $E+F$, where $E$ and $F$
form a pair of quasi-normed linear spaces
$$
\mathcal{E}(f, t)=\mathcal{E}(f, t,E, F)=\inf_{g\in F,
\|g\|_{F}\leq t}\|f-g\|_{E}.
$$

\begin{definition}The approximation space $\mathcal{E}_{\alpha,q}(E, F),
0<\alpha<\infty, 0<q\leq \infty $ is a quasi-normed linear spaces
of all $f\in E+F$ with the following quasi-norm
\begin{equation}
\left(\int_{0}^{\infty}\left(t^{\alpha}\mathcal{E}(f,
t)\right)^{q}\frac{dt}{t}\right)^{1/q}.
\end{equation}
\end{definition}

For a general quasi-normed linear spaces  $E$ the notation
$(E)^{\rho}$ is used for a quasi-normed linear spaces whose
quasi-norm is $\|\cdot\|^{\rho}$.

The following Theorem describes relations between interpolation
and approximation spaces (see  \cite{BL}, Ch. 7).
\begin{theorem}
If $\theta=1/(\alpha+1)$ and $r=\theta q, $ then
$$
(\mathcal{E}_{\alpha, r}(E, F))^{\theta}=(E,F)^{K}_{\theta,q}.
$$
\end{theorem}
The following important result is known as the Power Theorem (see
\cite{BL}, Ch. 7).
\begin{theorem}
Suppose that the following relations satisfied:  $\nu=\eta
\rho_{1}/\rho,$  $\rho=(1-\eta)\rho_{0}+\eta \rho_{1}, $ and
$q=\rho r$ for $ \rho_{0}>0, \rho_{1}>0.$ Then, if $0<\eta< 1, 0<
r\leq \infty,$ the following equality holds true
$$
\left((E)^{\rho_{0}}, (F)^{\rho_{1}}\right)^{K}_{\eta,
r}=\left((E,F)^{K}_{\nu, q}\right)^{\rho}.
$$
\end{theorem}

The Theorem we prove next represents a very abstract version of what is known as Direct and Inverse Approximation Theorems.

\begin{theorem}\label{A-I-equivalence}
 Suppose that $\mathcal{T}\subset F\subset E$ are quasi-normed
linear spaces and $E$ and $F$ are complete.
If there
exist $C>0$ and $\beta >0$ such that for any $f\in F$ the
following Jackson-type inequality is satisfied
\begin{equation}
t^{\beta}\mathcal{E}(t,f,\mathcal{T},E)\leq C\|f\|_{F},\label{dir}
t>0,
\end{equation}
 then the following embedding holds true
\begin{equation}\label{embed-2}
(E,F)^{K}_{\theta,q}\subset \mathcal{E}_{\theta\beta,q}(E,
\mathcal{T}), \>0<\theta<1, \>0<q\leq \infty.
\end{equation}

If there exist $C>0$ and $\beta>0$ such that for any $f\in \mathcal{T}$ the following Bernstein-type inequality holds
\begin{equation}\label{bern-1}
\|f\|_{F}\leq C\|f\|^{\beta}_{\mathcal{T}}\|f\|_{E},
\end{equation}
then 
\begin{equation}
\mathcal{E}_{\theta\beta, q}(E, \mathcal{T})\subset (F, F)^{K}_{\theta, q}, , \>0<\theta<1, \>0<q\leq \infty.
\end{equation}
\end{theorem}\label{intthm}

\begin{proof}

It is known (\cite{BL}, Ch.7)  that for any $s>0$, for
\begin{equation}
t=K_{\infty}(f,
s)=K_{\infty}(f,s,\mathcal{T},E)=\inf_{f=f_{1}+f_{2}, f_{1}\in
\mathcal{T}, f_{2}\in E}\max(\|f_{1}\|_{\mathcal{T}},
s\|f_{2}\|_{E})\label{t}
\end{equation}
the following inequality holds
\begin{equation}
s^{-1}K_{\infty}(f, s)\leq \lim_{\tau\rightarrow
t-0}\inf\mathcal{E}(f, \tau,E,\mathcal{T})\label{lim}.
\end{equation}
Since
\begin{equation}
K_{\infty}(f,s)\leq K(f, s)\leq 2K_{\infty}(f, s),\label{equiv}
\end{equation}
the Jackson-type inequality (\ref{dir}) and the inequality
 (\ref{lim}) imply
\begin{equation}
s^{-1}K(f,s,\mathcal{T},E)\leq Ct^{-\beta}\|f\|_{F}.
\end{equation}
The equality (\ref{t}), and inequality (\ref{equiv}) imply the
estimate
\begin{equation}
t^{-\beta}\leq 2^{\beta}
\left(K(f,s,\mathcal{T},E)\right)^{-\beta}
\end{equation}
which along with the previous inequality gives the estimate
$$
K^{1+\beta}(f,s,\mathcal{T},E)\leq C s \|f\|_{F}
$$
which in turn imply the inequality
\begin{equation}
K(f,s,\mathcal{T},E)\leq C s^{\frac{1}{1+\beta}}
\|f\|_{F}^{\frac{1}{1+\beta}}.\label{class1}
\end{equation}
At the same time one has
\begin{equation}
K(f, s,\mathcal{T},E)= \inf_{f=f_{0}+f_{1},f_{0}\in \mathcal{T},
f_{1}\in E}\left(\|f_{0}\|_{\mathcal{T}}+s\|f_{1}\|_{E}\right)\leq
s\|f\|_{E},\label{class2}
\end{equation}
for every $f$ in $E$. The inequality (\ref{class1}) means that the
quasi-normed linear space $(F)^{\frac{1}{1+\beta}}$ belongs to the
class $\mathcal{K}(\frac{1}{1+\beta}, \mathcal{T}, E)$ and
(\ref{class2}) means that the quasi-normed linear space $E$
belongs to the class $\mathcal{K}(1, \mathcal{T}, E)$. This fact
allows us to use the Reiteration Theorem to obtain the embedding
\begin{equation}
\left((F)^{\frac{1}{1+\beta}},E\right)^{K}_{\frac{1-\theta}{1+\theta
\beta},q(1+\theta \beta)}\subset \left(\mathcal{T},
E\right)^{K}_{\frac{1}{1+\theta \beta},q(1+\theta \beta)}
\end{equation}
for every $0<\theta<1, 1<q<\infty$. But the space on the left is
the space
$$
\left(E,(F)^{\frac{1}{1+\beta}}\right)^{K}_{\frac{\theta(1+\beta)}{1+\theta
\beta},q(1+\theta \beta)},
$$
which according to the Power Theorem is the space
$$
\left((E,F)^{K}_{\theta, q}\right)^{\frac{1}{1+\theta \beta}}.
$$
All these results along with the equivalence of  interpolation and
approximation spaces give  the embedding
$$
\left(E,F\right)^{K}_{\theta,q}\subset \left(\left(\mathcal{T},
E\right)^{K}_{\frac{1}{1+\theta \beta},q(1+\theta
\beta)}\right)^{1+\theta \beta}=\mathcal{E}_{\theta \beta,
q}(E,\mathcal{T}),
$$
which proves  the embedding (\ref{embed-2}). 
Conversely, if the Bernstein-type  inequality (\ref{bern-1}) holds then one has the inequality
\begin{equation}\label{bern-1}
\|f\|_{F}^{\frac{1}{1+\beta}}\leq C\|f\|^{\frac{\beta}{1+\beta}}_{\mathcal{T}}\|f\|_{E}^{\frac{1}{1+\beta}}.
\end{equation}
Along with obvious equality $\|f\|_{E}=\|f\|^{0}_{\mathcal{T}}\|f\|_{E}$ and the Iteration Theorem one obtains the embedding
$$ \left(\mathcal{T},
E\right)^{K}_{\frac{1}{1+\theta \beta},q(1+\theta
\beta)}\subset \left((F)^{\frac{1}{1+\beta}},E\right)^{K}_{\frac{1-\theta}{1+\theta
\beta},q(1+\theta \beta)}.
$$
To finish the proof of the theorem one can use the same arguments as above. Theorem is proven.

\end{proof}

\section{Approximation  in spaces $L_{p}(\mathbf{M}),\>\>1\leq p\leq \infty$}

In this section we are going to prove Theorem \ref{ApproxTh} by applying relations between Interpolation and approximation spaces described in Section \ref{interp}.

For $1 \leq p \leq \infty$, if $f \in L_p({\bf M})$, we let
\begin{equation}
\mathcal{E}(f,\omega,p)=\inf_{g\in
\textbf{E}_{\omega}(L)}\|f-g\|_{L_{p}({\bf M})}.
\end{equation}

\subsection{The Jackson inequality}

\begin{lemma}\label{Jack}
For every $k\in \mathbb{N}$ there exists a constant $C(k)$ such that for any $\omega>1$
$$
\mathcal{E}(f,\omega,p)\leq C(k)\omega^{-k}\|L^{k/2}f\|_{L_{p}({\bf M})}, \>\>\>f\in  W_{p}^{k}(\mathbf{M}).
$$

\end{lemma}
\begin{proof}

Let $h$ be a $C^{\infty}$ function on $[0,\infty)$ which equals $1$ on $[0,1]$,
and which is supported in $[0,4]$.  Define, for $\lambda > 0$,
\[ F(\lambda) = h(\lambda/4) - h(\lambda)\]
so that $F$ is supported in $[1,16]$.  For $j \geq 1$, we set 
\[ F_j(\lambda) = F(\lambda/4^{j-1}).\] 
We also set
$F_0 = h$, so that $\sum_{j=0}^{\infty} F_j \equiv 1$. 
For $\lambda > 0$ we define 
\[ \Psi(\lambda) = F(\lambda)/\lambda^{k/2} \]
so that $\Psi$ is supported in $[1,16]$.  For $j \geq 1$, we set 
\[ \Psi_j(\lambda)=\Psi(\lambda/4^{j-1}),\] 
so that 
$$
F_j(\lambda) = 2^{-(j-1)k}\Psi_j(\lambda)\lambda^{k/2}.
$$ 
Now for a given $\omega$ we change the variable $\lambda$ to the variable $4\lambda/\omega$. Clearly, the support of $h(4\lambda/\omega)$ is the interval $[0, \>\omega]$ and we have the following relation
$$
F_j(4\lambda/\omega) = \omega^{-k/2}2^{-(j-2)k}\Psi_j(4\lambda/\omega)\lambda^{k/2}.
$$
It implies that if $f$ is a distribution on ${\bf M}$ then
\[ 
F_j\left(\frac{4}{\omega}L\right)f =  \omega^{-k/2}2^{-(j-2)k}\Psi_j\left(\frac{4}{\omega}L\right)(L^{k/2}f), 
\]
in the sense of distributions.   
If now $f \in  W_p^k(\mathbf{M})$, so that $L^{k/2}f \in L_p(\mathbf{M})$, we see that for $\alpha'=\infty$ according to Theorem \ref{boundness}
$$
\left\|\Psi_j\left(\frac{4}{\omega}L\right)(L^{k/2}f)\right\|_{L_{p}({\bf M})}\leq C^{'}(k)\|L^{k/2}f\|_{L_{p}({\bf M})}, \>\>\> 
$$
which implies the inequality
$$
\mathcal{E}(f,\omega,p)\leq\left\|f-h\left(\frac{4}{\omega}L\right)f\right\|_{L_{p}({\bf M})}\leq \sum_{j\geq 1}\left\|F_{j}\left(\frac{4}{\omega}L\right)f\right\|_{L_{p}({\bf M})}\leq 
$$
$$
C^{'}(k)\omega^{-k/2}\sum_{j\geq 1}2^{-(j-2)k}\|L^{k/2}f\|_{L_{p}({\bf M})}\leq C(k)\omega^{-k/2}\|L^{k/2}f\|_{L_{p}({\bf M})}.
$$
The proof is complete.

\end{proof}

\subsection{The Bernstein inequality}

\begin{lemma}\label{BernLemma} There exists a constant $c(k)$ such that for all $f\in \textbf{E}_{\omega}(L)$
\begin{equation}\label{NI-1}
\|L^{k}f\|_{L_{p}({\bf M})}\leq C(k)\omega^{2k}\|f\|_{L_{p}({\bf M})}.
\end{equation}

\end{lemma}
\begin{proof}
Consider a function $h\in C_{0}^{\infty}(\mathbf{R}_{+})$ such that $h(\lambda)=1$ for $\lambda\in [0,1]$. 
For a fixed $\omega>0$ the support of $h( \lambda/\omega)$ is $[0,\omega]$ and it shows that for any $f\in \textbf{E}_{\omega}(L)$ one has the equality $h( \omega^{-1}L)f=f$.

  According to Theorems  \ref{nrdglc} and \ref{boundness} the operator $H(L)=(\omega^{-2}L)^{k}h(\omega^{-2}L)$ is bounded from $L_{p}(\mathbf{M})$ to $L_{q}(\mathbf{M})$. Thus for every $f\in \textbf{E}_{\omega}(L)$ we have
  $$
  \|L^{k}f\|_{L_{p}({\bf M})}=\omega^{2k} \|(\omega^{-2}L)^{k}h(\omega^{-2}L)f\|_{L_{p}({\bf M})}=
  $$
  $$
  \omega^{2k}\|H(L)f\|_{L_{p}({\bf M})}\leq C(k)\omega^{2k}\|f\|_{L_{p}({\bf M})}.
  $$

\end{proof}

\begin{remark}
In the inequality (\ref{NI-1}) the constant depends on the exponent $k$. One can show \cite{Pes08} that in the case of a compact homogeneous manifold a similar inequality holds with a constant that depends just on the manifold.
\end{remark}

\subsection{Besov spaces and approximations}

It is known  \cite{Tr}  that the Besov space $B^{\alpha}_{p,q}(\bold{M}), k\in \mathbb{N}, 1\leq p<\infty, 0\leq q \leq \infty, $ which was defined in (\ref{Besnorm}), is the interpolation space
$$
B^{\alpha}_{p,q}(\bold{M})=(L_{p}(\bold{M}),W^{r}_{p}(\bold{M}))^{K}_{\alpha/r,q},
$$
 where $K$ is  Peetre's interpolation functor.

Let us compare the situation on manifolds with the abstract conditions of the Theorem \ref{A-I-equivalence}. We treat linear normed spaces $W^{r}_{p}(\bold{M})$ and $L_{p}(\bold{M})$ as the spaces $E$ and $F$ respectively. We identify $\mathcal{T}$ with the linear space $\textbf{E}_{\omega}(L)$ which is equipped with the quasi-norm 
$$
\|f\|_{\mathcal{T}}= \inf_{\omega}\left\{\omega: f\in \textbf{E}_{\omega}(L)\right\},\>\>\>f\in \textbf{E}_{\omega}(L).
$$

Thus, Lemmas \ref{Jack}, \ref{BernLemma} and Theorem \ref{A-I-equivalence} imply the following result. 

\begin{theorem}
\label{contapproxim}
If  $\alpha  > 0$, $1 \leq p \leq \infty$, and $0 < q \leq  \infty$ then
$f \in B^{\alpha }_{p, q}({\bf M})$ if and only if $f \in L_p(\bold{M})$ and
\begin{equation}
\|f\|_{\mathcal{A}^{\alpha }_{p, q}({\bf M})} := 
\|f\|_{L_p(\bold{M})} +
\left(\int_{0}^{\infty}\left(t^{\alpha}\mathcal{E}(f,
t, p)\right)^{q}\frac{dt}{t}\right)^{1/q}<\infty.
\end{equation}
Moreover,
\begin{equation}
\label{errnrmeqway}
\|f\|_{\mathcal{A}^{\alpha }_{p, q}({\bf M})} \sim \|f\|_{B^{\alpha }_{p, q}({\bf M})}.
\end{equation}
\end{theorem}

By discretizing   the integral term (see \cite{BL})  we obtain  Theorem \ref{ApproxTh}.

 \makeatletter
\renewcommand{\@biblabel}[1]{\hfill#1.}\makeatother

\end{document}